\documentclass[11pt,a4paper]{article}
\usepackage{amsmath}
\usepackage{amssymb}
\usepackage{mathrsfs}

\numberwithin{equation}{section}  

\evensidemargin=\oddsidemargin\textwidth=142mm \textwidth=160mm
\newtheorem{theorem}{Theorem}[section]
\newtheorem{lemma}{Lemma}[section]
\newtheorem{proposition}{Proposition}[section]

\newtheorem{corollary}{Corollary}[section]

\newenvironment{proof}[1][Proof]{\begin{trivlist}
\item[\hskip \labelsep {\bfseries #1}]}{\end{trivlist}}

\begin{document}
\title {Generalized Hilbert Operator Acting on Bloch Type Spaces\footnote{ The research was supported by the National Natural Science Foundation of China (Grant No. 11671357, 11801508)}}
\author{ Shanli Ye\footnote{Corresponding author.~ E-mail address: slye@zust.edu.cn} \quad\quad Zhihui Zhou\footnote{E-mail address: zzhh144@163.com}  \\
\small \it School of Science, Zhejiang University of Science and Technology, Hangzhou 310023, China}

 \date{}
\maketitle
\begin{abstract}  Let $\mu$ be a positive Borel measure on the interval [0,1). For $\alpha>0$, the Hankel matrix $\mathcal{H}_{\mu,\alpha}=(\mu_{n,k,\alpha})_{n,k\geq 0}$ with entries $\mu_{n,k,\alpha}=\int_{[0,1)}\frac{\Gamma(n+\alpha)}{n!\Gamma(\alpha)}t^{n+k}d\mu(t)$ formally induces the operator
	$$\mathcal{H}_{\mu,\alpha}(f)(z)=\sum_{n=0}^{\infty}\left(\sum_{k=0}^{\infty} \mu_{n, k,\alpha} a_{k}\right)z^{n} $$ on the space of all analytic functions  $f(z)=\sum_{k=0}^{\infty}a_{k}z^{k}$ in the unit disc $\mathbb{D}$.
In this paper,
we characterize the measures $\mu$ for which $\mathcal{H}_{\mu,\alpha}$ ($\alpha\geq 2$) is a bounded (resp., compact) operator from the Bloch type space $\mathscr{B}_{\beta}$ ($0<\beta<\infty$) into $\mathscr{B}_{\alpha-1}$. We also give a necessary condition for which $\mathcal{H}_{\mu,\alpha}$ is a bounded operator by acting on Bloch type spaces for general cases.\\
{\small\bf Keywords}\quad {Hilbert operator, Bloch space, Bergman space
 \\
    {\small\bf 2020 MR Subject Classification }\quad 47B35, 30H30, 30H20\\}
\end{abstract}
\maketitle

\section{Introduction}
\hspace*{1.25em} Let $\mathbb{D}$ denote the open unit disk of the complex plane and $H(\mathbb{D})$ denote the set of all analytic functions in $\mathbb{D}$.

   For $\alpha>0$, the $\alpha$-Bloch space (also called Bloch type space), denoted by $\mathscr{B}_{\alpha}$, consists of those functions $f\in H(\mathbb{D})$ for which
$$\|f\|_{\mathscr{B}_{\alpha}}= \sup_{z\in\mathbb{D}}(1-|z|^2)^{\alpha}|f'(z)|<\infty.$$ The classical Bloch space $\mathscr{B}$ is just $\mathscr{B}_{1}$.
It is easy to check that $\mathscr{B}_{\alpha}$ equipped with the norm $\|f\|=|f(0)|+\|f\|_{\mathscr{B}_{\alpha}}$ is a Banach space.  We refer to \cite{Zhu_1993_Bloch} for more information of $\alpha$-Bloch space.

If $0<p<\infty$, the Bergman space $A^p$ consists of those functions $f\in H(\mathbb{D})$ satisfying
$$\|f\|_{A^p}^{p}=\int_{\mathbb{D}}|f(z)|^{p}dA(z)<\infty,$$
where $dA$ denotes the normalized Lebesgue area measure on $\mathbb{D}$. See \cite{duren_bergman_2004} for the theory of Bergman spaces.

Let us recall the definition of Carleson-type measures. If $s>0$ and $\mu$ is a positive Borel measure on $\mathbb{D}$. Then $\mu$ will be called an $s$-Carleson measure if there exists a positive constant $C$ such that
$$\mu(S(I))\leq C|I|^s$$
for every set $S(I)$ of the form
$$S(I)=\{z=re^{it}:~e^{it}\in I;~1-\frac{|I|}{2\pi}\leq r\leq 1\},$$
where $I$ is an interval of $\partial \mathbb{D}$ and $|I|$ denotes the length of $I$. If $\mu$ satisfies $\lim_{|I|\to 0}\frac{\mu(S(I))}{|I|^s}=0$, we say that $\mu$ is a vanishing s-Carleson measure. It is well known \cite{hastings_carleson_1975} that for $0<p\leq q<\infty$, $\mu$ is a $\frac{2q}{p}$-Carleson measure if and only if
there exists a positive constant $C$ such that the inequality
\begin{align}\label{eq1}	
	\left\{\int_{\mathbb{D}}|f(z)|^qd\mu(z)\right\}^{1/q}\leq C\|f\|_{A^p}
\end{align}
holds for all $f\in A^p.$

Let $\mu$ be a positive Borel measure on $\mathbb{D}$. For $0\leq \alpha<\infty$ and $0<s<\infty$, we say that $\mu$ is an $\alpha$-logarithmic $s$-Carleson measure, if there exists a positive constant $C$ such that
$$\frac{\mu(S(I))(\log\frac{2\pi}{|I|})^{\alpha}}{|I|^s}\leq C,\quad \mbox{for~all~interval}~I\subset\partial\mathbb{D}.$$
If $\mu(S(I))(\log\frac{2\pi}{|I|})^{\alpha}=o(|I|^s)$, as $|I|\to 0$, we say that $\mu$ is a vanishing $\alpha$-logarithmic $s$-Carleson measure (cf. \cite{zhao_Vanishing_logarithmic,zhao_logarithmic}).

A positive Borel measure $\mu$ on $[0,1)$ can be seen as a Borel measure on $\mathbb{D}$ by identifying it with the measure $\widetilde{\mu}$ defined as
$$\widetilde{\mu}(A)=\mu(A\cap[0,1)),$$ $\mbox{for~any~Borel~subset}~A~\mbox{of}~\mathbb{D}.$
In this way, we say that a positive Borel measure $\mu$ on [0,1) is an $s$-Carleson measure if and only if there exists a positive constant $C$ such that
$$\mu([t,1))\leq C(1-t)^s,\quad 0\leq t<1.$$
Also, we have similar statements for the other cases.

If $\mu$ is a positive Borel measure on $[0,1)$, for $\alpha>0$, we define $\mathcal{H}_{\mu,\alpha}=(\mu_{n,k,\alpha})_{n,k\geq 0}$ to be the Hankel matrix with entries $\mu_{n,k,\alpha}=\displaystyle\int_{[0,1)}\frac{\Gamma(n+\alpha)}{n!\Gamma(\alpha)}t^{n+k}d\mu(t)$. The matrix $\mathcal{H}_{\mu,\alpha}$ can be viewed as an operator on $H(\mathbb{D})$ by its action on the Taylor coefficients: $$a_n\to \sum_{k=0}^{\infty}\mu_{n,k,\alpha}a_k,\quad n=0,1,2,\cdots.$$ To be precise, if $f(z)=\sum_{k=0}^{\infty}a_kz^k\in H(\mathbb{D})$, we define the Hankel operator $\mathcal{H}_{\mu,\alpha}$ as
\begin{equation}\label{eq2}
	\mathcal{H}_{\mu,\alpha}(f)(z)=\sum_{n=0}^{\infty}\big(\sum_{k=0}^{\infty}\mu_{n,k,\alpha}a_k\big)z^n,
\end{equation}
whenever the right hand side makes sense and defines an analytic function in $\mathbb{D}$. The operators $\mathcal{H}_{\mu,1}$ have been extensively studied in \cite{Hankel_matrices_Dirichlet_2014,chatzifountas_generalized_2014,galanopoulos_hankel_2010,girela_generalized_2018,Li-2021}.
In this case, if we let $\mu$ be the Lebesgue measure on $[0,1)$, we can find that $\mathcal{H}_{\mu,1}$ is just the classical Hilbert matrix $\mathcal{H}=\big((n+k+1)^{-1}\big)$, which induces the classical Hilbert operator (see \cite{aleman_eigenfunctions_2012,diamantopoulos_composition_2000,galanopoulos_generalized_2014} for more details). For the case $\alpha=2$, we have studied the operator in \cite{Ye-Zhou-Bergman,Ye-Zhou-Bloch} and we call $\mathcal{H}_{\mu,2}$ the Derivative-Hilbert operator.  In this paper, we also call $\mathcal{H}_{\mu,\alpha}$ $(\alpha>0)$ the generalized Hilbert operator.

Galanopoulos and Pel$\acute{\text{a}}$ez proved in \cite{galanopoulos_hankel_2010} that the operator $\mathcal{H}_{\mu,1}$ is well defined in $H^1$ when $\mu$ is a Carleson measure. See \cite{duren_theory_1970} for more details on Hardy spaces. Also, they obtained the following integral representation
$$\mathcal{H}_{\mu,1}(f)(z)=\int_{[0,1)}\frac{f(t)}{1-tz}d\mu(t),\quad z\in \mathbb{D},~\mbox{for~all}~f\in H^1.$$

In \cite{Ye-Zhou-Bergman} and \cite{Ye-Zhou-Bloch}, we obtained the following integral representation $$\mathcal{H}_{\mu,2}f(z)=\int_{[0,1)}\frac{f(t)}{(1-tz)^{2}}d\mu (t),~ z\in\mathbb{D},$$ for all $f\in A^{p}$ ($0<p<\infty$) and for all $f\in \mathscr{B}$ respectively.

 For $\alpha>0$, we define the generalized integral-Hilbert operator
\begin{eqnarray}\label{eq3}
	\mathcal{I}_{\mu,\alpha}(f)(z)=\int_{[0,1)}\frac{f(t)}{(1-tz)^{\alpha}}d\mu (t),
\end{eqnarray}
whenever the right hand side makes sense and defines an analytic function in $\mathbb{D}$. Similar to the cases $\alpha=1$ and $\alpha=2$,in this paper, we can also obtain the operators $\mathcal{H}_{\mu,\alpha}$ and $\mathcal{I}_{\mu,\alpha}$ are closely related for all $\alpha>0$.

In this article we characterize those measures $\mu$ for which $\mathcal{H}_{\mu,\alpha}$ ($\alpha\geq 2$) is a bounded (resp., compact) operator from Bloch type space $\mathscr{B_{\beta}}$ ($\beta>0$) into $\mathscr{B}_{\alpha-1}$. We also give a necessary condition for general cases

As usual, throughout this paper, $C$ denotes a positive constant which depends only on the displayed parameters but not necessarily the same from one occurrence to the next.

\section{The operator $\mathcal{I}_{\mu,\alpha}$ acting on Bloch type spaces}
\hspace*{1.25em} In this section, we shall first characterize the measures $\mu$ for which the integral-Hilbert operator $\mathcal{I}_{\mu,\alpha}$ ($\alpha\geq 2$) is bounded (resp. compact) from Bloch type space $\mathscr{B_{\beta}}$ ($\beta>0$) into $\mathscr{B}_{\alpha-1}$. Firstly,
we shall give some auxiliary lemmas, which are needed in this section.
\begin{lemma}\label{le1}$\cite{Zhu_1993_Bloch}$
	If $0<\alpha<1$, then $f\in\mathscr{B}_{\alpha}\subset H^{\infty}$. If $\alpha>1$, then $f\in\mathscr{B}_{\alpha}$ if and only if $f(z)=O\big((1-|z|^2)^{1-\alpha}\big).$
\end{lemma}
\begin{lemma}$\cite{Zhu_1993_Bloch}$\label{le22}
	For any $\alpha>1$ and $z\in\mathbb{D}$ we have
	$$f(z)=(\alpha-1)\int_{\mathbb{D}}\frac{(1-|w|^2)^{\alpha-2}f(w)}{(1-z\overline{w})^{\alpha}}dA(w)$$
	if $f$ is an analytic function on $\mathbb{D}$ with
	$$\int_{\mathbb{D}}(1-|z|^2)^{\alpha-2}|f(z)|dA(z)<\infty.$$
\end{lemma}
\begin{proposition}\label{pro1}
		Suppose $\mu$ is a positive Borel measure on $[0,1)$ and $\beta>0$.
		\begin{itemize}
			\item [$(i)$] If $\beta\in (0,1)$, then for any given $f\in\mathscr{B}_{\beta}$, the integral in $(\ref{eq3})$ when $\alpha>0$ uniformly converges on any compact subset of $\mathbb{D}$ if and only if the measure $\mu$ is finite.
			\item [$(ii)$] If $\beta=1$, then for any given $f\in\mathscr{B}_{\beta}$, the integral in $(\ref{eq3})$ when $\alpha>0$ uniformly converges on any compact subset of $\mathbb{D}$ if and only if the measure satisfies $\int_{[0,1)}\log\frac{e}{1-t}d\mu(t)<\infty.$
			\item [$(iii)$] If $\beta>1$, then for any given $f\in\mathscr{B}_{\beta}$, the integral in $(\ref{eq3})$ when $\alpha>0$ uniformly converges on any compact subset of $\mathbb{D}$ if and only if the measure satisfies $\int_{[0,1)}\frac{1}{(1-t)^{\beta-1}}d\mu(t)<\infty.$
		\end{itemize}
\end{proposition}
\begin{proof}
(i)	We first assume that $\mu$ is a  finite positive Borel measure on $[0,1)$. By Lemma \ref{le1}, we obtain that for every $f\in\mathscr{B}_{\beta}$ $(0<\beta<1)$, $\alpha>0$, $0<r<1$ and $z$ with $|z|\leq r$,
	\begin{align*}
		\int_{[0,1)}\frac{|f(t)|}{|1-tz|^{\alpha}}d\mu(t)&\leq \frac{1}{(1-r)^\alpha}\int_{[0,1)}|f(t)|d\mu(t)\\
		&\leq C\frac{\|f\|_{\infty}}{(1-r)^\alpha}\int_{[0,1)}d\mu(t)\\
		&=\frac{C\mu([0,1))\|f\|_{\infty}}{(1-r)^\alpha}.
	\end{align*}
	This implies that the integral $\int_{[0,1)}\frac{f(t)}{(1-tz)^{\alpha}}d\mu(t)$ uniformly converges on any compact subset of $\mathbb{D}$ and the resulting function $\mathcal{I}_{\mu,\alpha}(f)$ is analytic in $\mathbb{D}$.
	
	Suppose now that the operator $\mathcal{I}_{\mu,\alpha}$ when $\alpha>0$ is well defined in the Bloch type space $\mathscr{B}_{\beta}$ ($0<\beta<1$). Take $f(z)=1\in \mathscr{B}_{\beta}$ and $z=0$. Then
	$$\mathcal{I}_{\mu,\alpha}(f)(0)=\int_{[0,1)}d\mu(t)$$
	is a complex number. Since $\mu$ is a positive measure, we get the desired result.
	
	Parts (ii) and (iii) can be proved similarly to the proceeding one. We shall omit the details. We will simply remark. In (ii), we use the fact that $$|f(z)|\leq C\|f\|_{\mathscr{B}}\log\frac{e}{1-|z|},$$
	for every $z\in\mathbb{D}$ and take the function $f(z)=\log\frac{e}{1-z}\in\mathscr{B}$. In (iii), we use Lemma \ref{le1} and take the function $g(z)=(1-z)^{1-\beta}\in\mathscr{B}_{\beta}$ ($\beta>1$).
	\end{proof}
\begin{proposition}\label{pro2}
	For $\alpha\geq 2$, $\beta>0$. Suppose that $\mu$ is the corresponding measure stated in Proposition \ref{pro1} such that the integral in $(\ref{eq3})$ uniformly converges on any compact subset of $\mathbb{D}$. Then for every $f\in\mathscr{B}_{\beta}$, $g\in A^1$, $0\leq r<1$, we have
	\begin{equation}\label{eq-bloch}
	\int_{\mathbb{D}}\overline{\mathcal{I}_{\mu,\alpha}(f)(rz)}g(rz)(1-|z|^2)^{\alpha-2}dA(z)=\int_{[0,1)}\overline{f(t)}g(r^2t)d\mu(t).
	\end{equation}
\end{proposition}
\begin{proof}
	 By the assumption of the measure $\mu$, we can obtain that there exists a positive constant $C$ such that for all $f\in\mathscr{B}_{\beta}$,
	$$\int_{[0,1)}|f(t)|d\mu(t)\leq  C\|f\|_{\mathscr{B}_{\beta}}.$$
	Hence, for every $f\in\mathscr{B}_{\beta}$, $g\in A^1$, $0\leq r<1$, we have
	\begin{align}\label{eq21}
		\begin{split}
			& \int_{\mathbb{D}} \int_{[0,1)}\left|\frac{f(t) g(r z)(1-|z|^2)^{\alpha-2}}{(1-r t z)^{\alpha}}\right| d \mu(t) dA(z) \\
			\leq & \frac{C\|f\|_{\mathscr{B}_{\alpha}}}{(1-r)^{\alpha}} \int_{\mathbb{D}}|g(r z)| dA(z) \\
			\leq & \frac{C\|f\|_{\mathscr{B}_{\alpha}}}{(1-r)^{\alpha}}\|g_r\|_{A^1}\leq\frac{C\|f\|_{\mathscr{B}_{\alpha}}}{(1-r)^{\alpha}}\|g\|_{A^1}<\infty,
		\end{split}
	\end{align}
	where $g_r$ denotes by $g_r(z)=g(rz)$, $z\in\mathbb{D}$.
	
	Since $\alpha\geq 2$, we obtain that for every $g\in A^1$, $$\int_{\mathbb{D}}(1-|z|^2)^{\alpha-2}|g(z)|dA(z)\leq\int_{\mathbb{D}}|g(z)|dA(z)=\|f\|_{A^1}.$$
	Now Lemma \ref{le22} together with (\ref{eq21}) and Fubini's theorem yield
	\begin{equation*}
		\begin{split}
			&\int_{\mathbb{D}} \overline{I_{\mu,\alpha}(f)(rz)} g(rz)(1-|z|^2)^{\alpha-2} d A(z)\\
			=&\int_{\mathbb{D}} \int_{[0,1)} \frac{\overline{f(t)} d \mu(t)}{(1-r t \bar{z})^{\alpha}} g(r z)(1-|z|^2)^{\alpha-2} d A(z)\\
			=&\int_{[0,1)} \int_{\mathbb{D}} \frac{g(r z)(1-|z|^2)^{\alpha-2}}{(1-r t \bar{z})^{\alpha}} d A(z) \overline{f(t)} d \mu(t) \\
			=&\int_{[0,1)} \overline{f(t)} g(r^2t) d \mu(t).
		\end{split}
	\end{equation*}
This finishes the proof.
\end{proof}
\begin{theorem}\label{th21}
	Suppose $\alpha\geq 2$, $0<\beta<1$. If $\mu$ is a finite positive Borel measure, then the following conditions are equivalent.
	\begin{itemize}
		\item [$(i)$] $\mathcal{I}_{\mu,\alpha}$ is a bounded operator from $\mathscr{B}_{\beta}$ into $\mathscr{B}_{\alpha-1}$.
		\item [$(ii)$] $\mathcal{I}_{\mu,\alpha}$ is a compact operator from $\mathscr{B}_{\beta}$ into $\mathscr{B}_{\alpha-1}$.
		\item [$(iii)$] The measure $\mu$ is a 2-Carleson measure.
	\end{itemize}
\end{theorem}
\begin{proof}
	(ii)$\Rightarrow$(i) is trivial.
	
	(iii)$\Rightarrow$(ii)
	 Assume that $\mu$ is a 2-Carleson measure. Let $\{f_{n}\}$ be any sequence with $\sup_{n}\|f\|_{\mathscr{B}_{\beta}}\leq 1$ and $\lim_{n\rightarrow \infty}f_{n}(z)=0$ on any compact subset of $\mathbb{D}$. Hence, by Lemma 3.2 in \cite{zhang-2005}, we obtain that $\sup_{z\in\mathbb{D}}|f_{n}(z)|\rightarrow 0$ as $n\rightarrow \infty$. Using (\ref{eq1}) and Proposition \ref{pro2}, we obtain that
	\begin{align*}
	\int_{\mathbb{D}} \overline{I_{\mu,\alpha}(f_n)(r z)} g(rz)(1-|z|^2)^{\alpha-2} d A(z)
    =&\int_{[0,1)} \overline{f_n(t)} g(r^2t) d \mu(t)\\
    \leq& \sup_{0<t<1}|f_n(t)|\int_{[0,1)}|g(r^2t)| d \mu(t) \\
       \leq&C\sup_{0<t<1}|f_n(t)|\|g\|_{A^1},
    \end{align*}
  for all $g\in A^1$. Thus,
   \begin{align}\label{eq22}
    \lim_{n\rightarrow \infty}	\int_{\mathbb{D}} \overline{I_{\mu,\alpha}(f_n)(r z)} g(rz)(1-|z|^2)^{\alpha-2} d A(z)=0
  \end{align}
   Let us recall the duality theorem in \cite{Zhu_1993_Bloch}: For $\alpha>0$, $(A^1)^*\cong \mathscr{B}_{\alpha}$ under the pairing
	\begin{align}\label{eq23}
		\langle f, g\rangle=\lim _{r \rightarrow 1^-} \int_{\mathbb{D}} \overline{f\left(r z\right)} g\left(rz\right) (1-|z|^2)^{\alpha-1}d A(z),\quad f\in \mathscr{B}_{\alpha},~g\in A^1.
	\end{align}
This together with (\ref{eq22}) imply that $\mathcal{I}_{\mu,\alpha}(f_n)\rightarrow 0$ in $\mathscr{B}_{\alpha-1}$ as $n\rightarrow\infty$. So $\mathcal{I}_{\mu,\alpha}$ is a compact operator from $\mathscr{B}_{\beta}$ into $\mathscr{B}_{\alpha-1}$.
	
	(i)$\Rightarrow$(iii)	Suppose $\mathcal{I}_{\mu,\alpha}$ is a bounded operator from $\mathscr{B}_{\beta}$ into $\mathscr{B}_{\alpha-1}$. For $0<a<1$, set
	$$
	f_{a}(z)=1~\text{and}~ g_{a}(z)=\left(\frac{1-a^{2}}{(1-a z)^{2}}\right)^{2}, \quad z \in \mathbb{D}.
	$$
	Then $f_a(z)\in \mathscr{B}_{\beta},~g_a(z)\in A^1$ and
	\begin{align*}
		\sup_{0<a<1}\|f_a\|_{\mathscr{B}_{\beta}}\leq 2,\quad \sup_{0<a<1}\|g_a\|_{A^1}=1.
	\end{align*}
    Since $\mathcal{I}_{\mu,\alpha}$ is a bounded operator from $\mathscr{B}_{\beta}$ into $\mathscr{B}_{\alpha-1}$. It follows from (\ref{eq23}) and Proposition \ref{pro2} that there exists a positive constant $C$ such that
	\begin{align*}
		\left|\int_{[0,1)} \overline{f(t)} g(r^2t) d \mu(t)\right|\leq C\|f\|_{\mathscr{B}_{\beta}}\|g\|_{A^1},\quad 0<r<1,~f\in\mathscr{B}_{\beta},~g\in A^1.
	\end{align*}
	Taking $r\in[a,1)$, we obtain
	\begin{align*}
		\infty&>C\sup_{0<a<1}\|f_a\|_{\mathscr{B}_{\beta}}\sup_{0<a<1}\|g_a\|_{A^1}\\
		& \geq\left|\int_{[0,1)} f_{a}(t) g_{a}(r^2t) d \mu(t)\right| \\
		& \geq \int_{a}^{1}\left(\frac{1-a^{2}}{(1-ar^2 t)^{2}}\right)^{2} d\mu(t) \\
		& \geq \frac{C_1}{\left(1-a^{2}\right)^{2}} \mu([a, 1)).
	\end{align*}
	Therefore, $\mu$ is a 2-Carleson measure. We complete the proof.
\end{proof}
\begin{theorem}\label{th22}
	For $\alpha\geq 2$, if $\mu$ is a positive Borel measure with $\int_{[0,1)}\log\frac{e}{1-t}d\mu(t)<\infty.$
	\begin{itemize}
		\item[$(i)$] $\mathcal{I}_{\mu,\alpha}$ is a bounded operator from $\mathscr{B}$ into $\mathscr{B}_{\alpha-1}$ if and only if $\mu$ is a 1-logarithmic 2-Carleson measure.
		\item[$(ii)$] $\mathcal{I}_{\mu,\alpha}$ is a compact operator from $\mathscr{B}$ into $\mathscr{B}_{\alpha-1}$ if and only if $\mu$ is a vanishing 1-logarithmic 2-Carleson measure.
	\end{itemize}
\end{theorem}
\begin{proof}
	(i) Assume that $\mu$ is a 1-logarithmic 2-Carleson measure and let $d\nu(t)=\log\frac{2}{1-t}d\mu(t)$. Using Proposition 2.5 of \cite{girela_generalized_2018}, we obtain that $\nu$ is a 2-Carleson measure.
	
	 Applying (\ref{eq-bloch}) and Proposition \ref{pro2}, we have
	\begin{align*}
		\begin{split}
			\int_{\mathbb{D}} \overline{I_{\mu,\alpha}(f)(r z)} g(r z) (1-|z|^2)^{\alpha-2}d A(z)
			=&\int_{[0,1)} \overline{f(t)} g(r^2t) d \mu(t)\\
			\leq& C\|f\|_{\mathscr{B}} \int_{[0,1)}|g(r^2t)| \log \frac{2}{1-t} d \mu(t) \\
			=&C\|f\|_{\mathscr{B}} \int_{[0,1)}|g(r^2t)| d\nu(t)\\
			\leq& C\|f\|_{\mathscr{B}}\|g\|_{A^1},\quad f\in\mathscr{B},~g\in A^1.
		\end{split}
	\end{align*}
Then (\ref{eq23}) shows that $\mathcal{I}_{\mu,\alpha}$ is a bounded operator from $\mathscr{B}$ into $\mathscr{B}_{\alpha-1}$.
	
	Conversely, suppose that $\mathcal{I}_{\mu,\alpha}$ is a bounded operator from $\mathscr{B}$ into $\mathscr{B}_{\alpha-1}$. By (\ref{eq23}), we have
		\begin{align*}
		\left|\int_{[0,1)} \overline{f(t)} g(r^2t) d \mu(t)\right|\leq C\|f\|_{\mathscr{B}}\|g\|_{A^1},\quad 0<r<1,~f\in\mathscr{B},~g\in A^1.
	\end{align*}
 For $0<a<1$, we set
	$$
	f_{a}(z)=\log \frac{e}{1-a z}~\text{and}~ g_{a}(z)=\left(\frac{1-a^{2}}{(1-a z)^{2}}\right)^{2}, \quad z \in \mathbb{D}.$$
  Then
	\begin{align*}
		\sup_{0<a<1}\|f_a\|_{\mathscr{B}}\leq 3,\quad \sup_{0<a<1}\|g_a\|_{A^1}=1.
	\end{align*}
	Taking $r\in[a,1)$, we deduce that
	\begin{align*}
		\infty&>C\sup_{0<a<1}\|f_a\|_{\mathscr{B}}\sup_{0<a<1}\|g_a\|_{A^1}\\
		& \geqslant\left|\int_{[0,1)} f_{a}(t) g_{a}(r^2t) d \mu(t)\right| \\
		& \geqslant \int_{a}^{1}\left(\frac{1-a^{2}}{(1-ar^2 t)^{2}}\right)^{2} \log \frac{e}{1-a t} d \mu(t) \\
		& \geqslant C_1\frac{\log \frac{e}{1-a^{2}}}{\left(1-a^{2}\right)^{2}} \mu([a, 1)).
	\end{align*}
	This implies that $\mu$ is a 1-logarithmic 2-Carleson measure.
	
	(ii) First suppose that $\mu$ is a vanishing 1-logarithmic 2-Carleson measure. Let $\{f_n\}_{n=1}^{\infty}$ be a bounded sequence in the Bloch space which converges to $0$ uniformly on any compact subset of $\mathbb{D}$. Then
	\begin{align}\label{eq25}
		\lim_{n\rightarrow\infty}\int_{[0, r)}\left|f_{n}(t)\right||g(t)| d \mu(t)=0,\quad g\in A^1.
	\end{align}
	Also, let $d\nu(t)=\log\frac{2}{1-t}d\mu(t)$. Then $\nu$ is a vanishing 2-Carleson. For $0<r<1$, we set
	$$d\nu_r(z)=\chi_{r<|z|<1}(t)d\nu(t).$$ Therefore,
	\begin{align}\label{eq26}
		\int_{[r, 1)}\left|f_{n}(t)\right||g(t)| d \mu(t)\leq \int_{[0,1)}|g(t)| d \nu_{r}(t)<C \mathcal{N}\left(\nu_{r}\right)\|g\|_{A^{1}},\quad g\in A^1,
	\end{align}
	where $\mathcal{N}(\mu)$ is the norm of identity mapping $i$ from $A^1$ into $L^1(\mathbb{D},~\mu)$. It is well known that $\mathcal{N}\left(\nu_{r}\right)\rightarrow 0$ as $r\rightarrow 1^{-}$ if and only if $\nu$ is a vanishing 2-Carleson measure. This together with (\ref{eq25}), (\ref{eq26}) imply that
	$$
	\lim _{n \rightarrow \infty} \int_{[0,1)}\left|f_{n}(t) \| g(t)\right| d \mu(t)=0, \quad \text { for all } g \in A^{1}.
	$$
	Hence Proposition \ref{pro2} shows
	$$
	\lim _{n \rightarrow \infty}\left(\lim _{r \rightarrow 1} \left\vert \int_{\mathbb{D}} \overline{I_{\mu,\alpha}(f_{n}) (rz)} g\left(rz\right) (1-|z|^2)^{\alpha-2}d A(z) \right\vert\right)=0, \quad \text { for all } g \in A^{1}.
	$$
	Thus by (\ref{eq23}), it follows that $\mathcal{I}_{\mu,\alpha}(f_n)\rightarrow 0$ in $\mathscr{B}_{\alpha-1}$. So $\mathcal{I}_{\mu,\alpha}$ is a compact operator from $\mathscr{B}$ into $\mathscr{B}_{\alpha-1}$.
	
	Suppose now that $\mathcal{I}_{\mu,\alpha}$ is a compact operator from $\mathscr{B}$ into $\mathscr{B}_{\alpha-1}$. Let $\{a_n\}\subset (0,1)$ be any sequence with $a_n\rightarrow 1$. We set $$f_{a_n}(z)=\frac{1}{\log\frac{2}{1-a_n^2}}\left(\log\frac{2}{1-a_{n}z}\right)^2.$$
	A calculation shows that $f_{a_n}\in\mathscr{B}$, $\sup_{n\geq 1}\|f_{a_n}\|_{\mathscr{B}}<\infty$ and $\{f_{a_n}\}$ is bounded  uniformly sequence in
	$\mathscr{B}$ which converges to $0$ uniformly on every compact subset of $\mathbb{D}$. Then $\{\mathcal{I}_{\mu_2}(f_{a_n})\}$ converges to 0 in $\mathscr{B}_{\alpha-1}$. We obtain that for every $g\in A^1$,
	\begin{align}\label{eq27}
		\begin{split}
			&\lim_{n\rightarrow \infty}\int_{[0,1)} \overline{f_{a_n}(t)} g_{a_n}(r^2t) d \mu(t)\\
			=&\lim_{n\rightarrow \infty}\int_{\mathbb{D}} \overline{I_{\mu,\alpha}(f_{a_n})(rz)}g_{a_n}(rz)(1-|z|^2)^{\alpha-2}d A(z)=0,\quad 0<r<1.
		\end{split}
	\end{align}
	We also set $$g_{a_n}(z)=\left(\frac{1-a_n^{2}}{(1-{a_n}z)^{2}}\right)^{2}\in A^1.$$
	For $r\in (a_n,1),$ we deduce that
	\begin{align*}
		&\int_{[0,1)} \overline{f_{a_n}(t)} g_{a_n}(r^2t) d \mu(t) \\
		\geq& \int_{a_n}^{1}\left(\frac{1-a_n^{2}}{(1-{a_n} r^2t)^{2}}\right)^{2} \frac{1}{\log\frac{2}{1-a_n^2}}\left(\log\frac{2}{1-a_{n}t}\right)^2 d\mu(t) \\
		\geq& C\frac{\log \frac{2}{1-a_n^{2}}}{\left(1-a_n^{2}\right)^{2}} \mu([a_n, 1)).
	\end{align*}
	Letting $n\rightarrow \infty$ and bearing in mind that $\{a_n\}\subset (0,1)$ is a sequence with $a_n\rightarrow 1$, \eqref{eq27} gives
	$$\lim_{a\rightarrow 1^-}\frac{\log \frac{2}{1-a}}{\left(1-a\right)^{2}} \mu([a, 1))=0,$$
	which is equivalent to saying that $\mu$ is a vanishing 1-logarithmic 2-Carleson measure. This completes the proof.
\end{proof}
\begin{theorem}\label{th23}
	For $\alpha\geq 2$, $\beta>1$, if $\mu$ is a positive Borel measure with $\int_{[0,1)}\frac{1}{(1-t)^{\beta-1}}d\mu(t)<\infty.$
	\begin{itemize}
		\item[$(i)$] $\mathcal{I}_{\mu,\alpha}$ is a bounded operator from $\mathscr{B}_{\beta}$ into $\mathscr{B}_{\alpha-1}$ if and only if $\mu$ is a ($\beta$+1)-Carleson measure.
		\item[$(ii)$] $\mathcal{I}_{\mu,\alpha}$ is a compact operator from $\mathscr{B}$ into $\mathscr{B}_{\alpha-1}$ if and only if $\mu$ is a vanishing ($\beta$+1)-Carleson measure.
	\end{itemize}
\end{theorem}
\begin{proof}
	(i) Suppose that $\mu$ is a ($\beta+1$)-Carleson measure and let $d\nu(t)=(1-t)^{1-\beta}d\mu(t)$. Using Lemma 3.2 of \cite{Li-2021}, we obtain that $\nu$ is a 2-Carleson measure. Then (\ref{le1}) and Proposition \ref{pro2} imply that
	\begin{align*}
		\begin{split}
			\int_{\mathbb{D}} \overline{I_{\mu,\alpha}(f)(r z)} g(r z) (1-|z|^2)^{\alpha-2}d A(z)
			=&\int_{[0,1)} \overline{f(t)} g(r^2t) d \mu(t)\\
			\leq& C\|f\|_{\mathscr{B}_{\beta}} \int_{[0,1)}|g(r^2t)| (1-t)^{1-\beta} d \mu(t) \\
			=&C\|f\|_{\mathscr{B}_{\beta}} \int_{[0,1)}|g(r^2t)| d\nu(t)\\
			\leq& C\|f\|_{\mathscr{B}_{\beta}}\|g\|_{A^1},\quad f\in\mathscr{B}_{\beta},~g\in A^1.
		\end{split}
	\end{align*}
This and (\ref{eq23}) show that $\mathcal{I}_{\mu,\alpha}$ is a bounded operator from $\mathscr{B}_{\beta}$ into $\mathscr{B}_{\alpha-1}$.

On the other hand, suppose $\mathcal{I}_{\mu,\alpha}$ is a bounded operator from $\mathscr{B}_{\beta}$ into $\mathscr{B}_{\alpha-1}$. For $0<a<1$, set
$$
f_{a}(z)=\frac{1-a^2}{(1-az)^{\beta}}$$
and
 $$g_{a}(z)=\left(\frac{1-a^{2}}{(1-a z)^{2}}\right)^{2}.$$
 It is easy to check that there exists a positive constant $C$ which is dependent only on $\beta$ such that
	\begin{align*}
	\sup_{0<a<1}\|f_a\|_{\mathscr{B}_{\beta}}\leq C,\quad \sup_{0<a<1}\|g_a\|_{A^1}=1.
\end{align*}
Now the proof is similar to the proof of Theorem \ref{th22} (i).

(ii) For $0<a<1$, Let $f_a$ be defined as (i). Then $\{f_a\}$ is bounded sequence in $\mathscr{B}_{\beta}$ and $\lim_{a\rightarrow 1}f_{a}(z)=0$ on any compact subset of $\mathbb{D}$. From now on, the proof is analogous to the proof of Theorem \ref{th22} (ii) and we omit the details.
\end{proof}
\section{The operator $\mathcal{H}_{\mu,\alpha}$ acting on Bloch type spaces}
\hspace*{1.25em} In this section, we obtain the relationship between $\mathcal{H}_{\mu,\alpha}$ and $\mathcal{I}_{\mu,\alpha}$. We also obtain a necessary condition such that $\mathcal{H}_{\mu,\alpha}$ is a bounded operator acting on Bloch type spaces for general cases. Then we characterize the measures $\mu$ for which the operator $\mathcal{H}_{\mu,\alpha}$ ($\alpha\geq 2$) is bounded (resp. compact) from Bloch type spaces into $\mathscr{B}_{\alpha-1}$. We begin with describing for which measures $\mu$ the operator $\mathcal{H}_{\mu,\alpha}$ is well defined on the Bergman space $A^p$ ($0<p<\infty$), which is a generalization of Theorem 2.1 in \cite{Ye-Zhou-Bergman}.
\begin{proposition}\label{pro3.1}
	Suppose $0<p<\infty$, $\alpha>0$ and let $\mu$ be a positive Borel measure on $[0,1)$. Then the power series in $(\ref{eq2})$ defines an analytic function in $\mathbb{D}$ for every $f\in A^p$ in any of the three following cases.	
	\begin{itemize}
		\item[$(i)$] $\mu$ is a $2/p$-Carleson measure if $0<p\leq 1$.
		\item[$(ii)$] $\mu$ is a $\left(\frac{2-(p-1)^2}{p}\right)$-Carleson measure if $1\leq p\leq 2$.
		\item[$(iii)$] $\mu$ is a $1/p$-Carleson measure if $2\leq p<\infty$.
	\end{itemize}
	Furthermore, in such cases we have that
	\begin{align}\label{eq2.1}
		\mathcal{H}_{\mu,\alpha}(f)(z)=\mathcal{I}_{\mu,\alpha}(f)(z),\quad z\in\mathbb{D},~f\in A^p.
	\end{align}
\end{proposition}
\begin{proof}
(i)  Suppose that $\mu$ is a $2/p$-Carleson measure. By Proposition 1 in \cite{chatzifountas_generalized_2014},  we obtain that there exists a positive constant $C$ which is independent of $n,~k$ such that
$$|\mu_{n,k,\alpha}|\leq C\frac{\Gamma(n+\alpha)}{n!\Gamma(\alpha)}\frac{1}{(k+1)^{2/p}}.$$
This and Theorem 4 in \cite[p.85]{duren_bergman_2004} imply that, for every $n$,
\begin{align*}
	\sum_{k=0}^{\infty}\left|\mu_{n, k,\alpha}\right|\left|a_{k}\right| & \leq C\frac{\Gamma(n+\alpha)}{n!\Gamma(\alpha)} \sum_{k=0}^{\infty} \frac{\left|a_{k}\right|^{p}\left|a_{k}\right|^{1-p}}{(k+1)^{2/p}} \\
	& \leq C\frac{\Gamma(n+\alpha)}{n!\Gamma(\alpha)} \sum_{k=0}^{\infty} \frac{\left|a_{k}\right|^{p}\left|k+1\right|^{\frac{(2-p)(1-p)}{p}}}{(k+1)^{2/p}}\\
	&\leq C\frac{\Gamma(n+\alpha)}{n!\Gamma(\alpha)} \sum_{k=0}^{\infty} (k+1)^{p-3}\left|a_{k}\right|^{p}.
\end{align*}
This converges by Theorem 3 in \cite[p.83]{duren_bergman_2004}. So the power series in (\ref{eq2}) is well defined analytic function in $\mathbb{D}$ and
$$\sum_{k=0}^{\infty} \mu_{n, k,\alpha} a_{k}=\int_{[0,1)}\frac{\Gamma(n+\alpha)}{n!\Gamma(\alpha)}t^n f(t)d\mu(t).$$
On the other hand, notice that for each $z\in\mathbb{D}$,
\begin{align*}
	&\sum_{n=0}^{\infty}(\int_{[0,1)}\frac{\Gamma(n+\alpha)}{n!\Gamma(\alpha)}t^n |f(t)|d\mu(t))|z|^n\\
	\leq& \sum_{n=0}^{\infty}\left(\sum_{k=0}^{\infty} \mu_{n, k}|a_{k}|\right)|z|^n
	\leq C \sum_{k=0}^{\infty}\left|a_{k}\right|^{p} (k+1)^{p-3}\sum_{n=0}^{\infty}\frac{\Gamma(n+\alpha)}{n!\Gamma(\alpha)}|z|^n\\
	\leq &\frac{C}{(1-|z|)^{\alpha}}\sum_{k=0}^{\infty}\left|a_{k}\right|^{p} (k+1)^{p-3}<\infty.
\end{align*}
Therefore, we have
\begin{align*}
	\mathcal{H}_{\mu,\alpha}(f)(z) &=\sum_{n=0}^{\infty}\left(\int_{[0,1)}\frac{\Gamma(n+\alpha)}{n!\Gamma(\alpha)} t^{n} f(t) d \mu(t)\right) z^{n} \\
	&=\int_{[0,1)} \sum_{n=0}^{\infty}\frac{\Gamma(n+\alpha)}{n!\Gamma(\alpha)}(t z)^{n} f(t) d \mu(t) \\
	&=\int_{[0,1)} \frac{f(t)}{(1-t z)^{\alpha}} d \mu(t)=	\mathcal{I}_{\mu,\alpha}(f)(z),
\end{align*}
for every $z\in\mathbb{D}$ and $f\in A^p$.
The proof of (ii) and (iii) is analogous to the proof of (i). We omit the details.
\end{proof}

Notice that $\mathscr{B}_{\beta}\subset\cap_{0<p<\infty}A^{p}$ for $0<\beta\leq 1$. We can easily obtain the following corollary by Proposition \ref{pro3.1}
\begin{corollary}\label{cor31}
	Suppose $\alpha>0$, $0<\beta\leq 1$. If $\mu$ is a $t$-Carleson measure for some $t>0$, then for every $f\in\mathscr{B}_{\beta}$ and $z\in\mathbb{D}$,
	$$	\mathcal{H}_{\mu,\alpha}(f)(z)=\mathcal{I}_{\mu,\alpha}(f)(z).$$
\end{corollary}
\begin{proposition}\label{pro32}
	Suppose $\alpha,\beta>0$, if $\mu$ is a $\beta$-Carleson measure, then for every $f\in\mathscr{B}_{\beta}$ and $z\in\mathbb{D}$,
	$$	\mathcal{H}_{\mu,\alpha}(f)(z)=\mathcal{I}_{\mu,\alpha}(f)(z).$$
	\end{proposition}

To proof the Proposition \ref{pro32}, we still need the following lemma, which can be found in \cite{Li-2021}.
\begin{lemma}\label{le31}
	Let $f(z)=\sum_{n=0}^{\infty}a_{n}z^{n}\in\mathscr{B}_{\alpha}$ for any $\alpha>0$. Then
	$$\sup_{n}\sum_{k=2^{n}+1}^{2^{n+1}}\big|\frac{a_{k}}{k^{\alpha-1}}\big|^{2}<C\|f\|_{\mathscr{B}_{\alpha}}^{2}.$$
\end{lemma}
\textbf{Proof of Proposition \ref{pro32}.} Since $\mu$ is a $\beta$-Carleson measure, by Proposition 1 in \cite{chatzifountas_generalized_2014}, we have $$|\mu_{n,k,\alpha}|\leq C\frac{\Gamma(n+\alpha)}{n!\Gamma(\alpha)}\frac{1}{(k+1)^{\beta}}.$$
Notice that $\frac{\Gamma(n+\alpha)}{n!\Gamma(\alpha)}\sim n^{\alpha-1}$ as $n\rightarrow \infty$ by Stirling's formula. It follows that there exist a constant $C>0$ such that
\begin{align*}
	\sum_{n=0}^{\infty}\left(\sum_{k=2^{n}+1}^{2^{n+1}}\big|\frac{\mu_{n,k,\alpha}}{k^{1-\beta}}\big|^{2}\right)^{1/2}\leq C\sum_{n=0}^{\infty}\left(n^{\alpha-1}\sum_{k=2^{n}+1}^{2^{n+1}}\frac{1}{k^2}\right)^{1/2}\leq C\sum_{n=0}^{\infty}\frac{n^{(\alpha-1)/2}}{2^{n/2}}<\infty.
\end{align*}
This together with Lemma \ref{le31} yield that for every $f\in\mathscr{B}_{\beta}$,
\begin{align*}
	\sum_{k=1}^{\infty}|\mu_{n,k,\alpha}a_k|&\leq\sum_{k=1}^{\infty}\frac{\mu_{n,k,\alpha}}{k^{1-\beta}}\frac{|a_{k}|}{k^{\beta-1}}\\
	&\leq\sum_{n=0}^{\infty}\left(\sum_{k=2^{n}+1}^{2^{n+1}}\big|\frac{\mu_{n,k,\alpha}}{k^{\beta-1}}\big|^{2}\right)^{1/2}\sup_{n}\left(\sum_{k=2^{n}+1}^{2^{n+1}}\big|\frac{a_{k}}{k^{\beta-1}}\big|^{2}\right)^{1/2}\\
	&<C\|f\|_{\mathscr{B}_{\beta}}.
\end{align*}
From now on, arguing as in Proposition \ref{pro3.1} (i), we can obtain the desired result.

Our next objective is to give a necessary condition such that the operator $\mathcal{H}_{\mu,\alpha}$ is bounded on Bloch type spaces for general cases.

\begin{theorem}\label{th35}
	Suppose $\alpha,\beta,\gamma>0$ and $\mathcal{H}_{\mu,\alpha}$ is a bounded operator from $\mathscr{B}_{\beta}$ into $\mathscr{B}_{\gamma}$.
	\begin{itemize}
		\item [($i$)] If $1<\beta<\infty$, then the measure $\mu$ is an $(\alpha+\beta-\gamma-1/2)$-Carleson measure if $\alpha+\beta-\gamma-1/2>0$ and $\mu$ is finite if $\alpha+\beta-\gamma-1/2= 0$.
		\item [($ii$)] If $0<\beta<1$, then the measure $\mu$ is an $(\alpha-\gamma+1/2)$-Carleson measure if $\alpha-\gamma+1/2>0$ and $\mu$ is finite if $\alpha-\gamma+1/2= 0$.
	\end{itemize}
\end{theorem}
\begin{proof}
(i) Assume that the operator $\mathcal{H}_{\mu,\alpha}$ is a bounded operator from $\mathscr{B}_{\beta}$ into $\mathscr{B}_{\gamma}$. For any $0<\lambda<1$, we set
$$f_{\lambda}(z)=\frac{1-\lambda^{2}}{(1-\lambda z)^{\beta}}=\sum_{k=0}^{\infty}a_{k,\lambda}z^{k}\in\mathscr{B}_{\beta}.$$
Then $$\mathcal{H}_{\mu,\alpha}(f)(z)=\sum_{n=0}^{\infty}\big(\sum_{k=0}^{\infty}\mu_{n,k,\alpha}a_{k,\lambda}\big)z^n\in \mathscr{B}_{\gamma}.$$
 It is easy to check that $a_{k,\lambda}\sim((1-\lambda^2)k^{\beta-1}\lambda^{k}).$
 So Lemma \ref{le31} gives
 \begin{align*}
 	\infty&>\sup_{j}\sum_{n=2^j+1}^{2^{j+1}}\left|\frac{\sum_{k=0}^{\infty}\mu_{n,k,\alpha}a_{k,\lambda}}{n^{\gamma-1}}\right|^2\\
 	&\geq C\sup_{j}\sum_{n=2^j+1}^{2^{j+1}}\left|\frac{\sum_{k=0}^{\infty}(1-\lambda^2)k^{\beta-1}\lambda^{k}n^{\alpha-1}\int_{0}^{1}t^{n+k}d\mu(t)}{n^{\gamma-1}}\right|^2\\
 	&\geq C\sup_{j}\sum_{n=2^j+1}^{2^{j+1}}\left|\frac{\sum_{k=0}^{\infty}(1-\lambda^2)k^{\beta-1}\lambda^{2k+n}n^{\alpha-1}\mu([\lambda,1))}{n^{\gamma-1}}\right|^2\\
 	&\geq C\sup_{j}\sum_{n=2^j+1}^{2^{j+1}}\left|(1-\lambda^2)\lambda^{n}n^{\alpha-\gamma}\mu([\lambda,1))\sum_{k=0}^{\infty}k^{\beta-1}\lambda^{2k}\right|^2\\
 	&\geq
 	C\sup_{j}\sum_{n=2^j+1}^{2^{j+1}}\left|\frac{(1-\lambda^2)\lambda^{n}n^{\alpha-\gamma}\mu([\lambda,1))}{(1-\lambda^2)^{\beta}}\right|^{2}\\
 &	\ge	C\sup_{j}2^{j}\left|\frac{(1-\lambda^2)\lambda^{2^{j+1}}(2^{j})^{\alpha-\gamma}\mu([\lambda,1))}{(1-\lambda^2)^{\beta}}\right|^{2}
 \end{align*}
Choosing $j$ such that $2^j\leq\frac{1}{1-\lambda}<2^{j+1}$, we have
$$\infty>\frac{C\mu([\lambda,1))}{(1-\lambda^2)^{\alpha+\beta-\gamma-1/2}}.$$
This is equivalent to saying that the measure $\mu$ is an $(\alpha+\beta-\gamma-1/2)$-Carleson measure if $\alpha+\beta-\gamma-1/2>0$ and $\mu$ is finite if $\alpha+\beta-\gamma-1/2\leq 0$. This completes the proof of Theorem \ref{th35}.

(ii) The proof is similar to the proceeding one by taking $f_{\lambda}=1$.
\end{proof}
\begin{corollary}\label{cor32}
     Suppose $\alpha,\beta>0$ and $\mathcal{H}_{\mu,\alpha}$ is a bounded operator from $\mathscr{B}_{\beta}$ into $\mathscr{B}_{\alpha-1}$, then the measure $\mu$ is a $3/2$-Carleson measure for $0<\beta<1$ and $\mu$ is an $(\beta+1/2)$-Carleson measure for $1<\beta<\infty$.
\end{corollary}

Now we can give the main results in this paper. By using Corollaries \ref{cor31} and \ref{cor32} together with Theorem \ref{th21} and Theorem \ref{th22} respectively, we can obtain the following two theorems.

\begin{theorem}
	Suppose $\alpha\geq 2$, $0<\beta<1$. Let $\mu$ be a positive Borel measure on $[0,1)$, then the following conditions are equivalent.
	\begin{itemize}
		\item [$(i)$] $\mathcal{H}_{\mu,\alpha}$ is a bounded operator from $\mathscr{B}_{\beta}$ into $\mathscr{B}_{\alpha-1}$.
		\item [$(ii)$] $\mathcal{H}_{\mu,\alpha}$ is a compact operator from $\mathscr{B}_{\beta}$ into $\mathscr{B}_{\alpha-1}$.
		\item [$(iii)$] The measure $\mu$ is a 2-Carleson measure.
	\end{itemize}
\end{theorem}
\begin{theorem}
	Suppose $\alpha\geq 2$. Let $\mu$ be a positive Borel measure on $[0,1)$, then
	\begin{itemize}
		\item[$(i)$] $\mathcal{H}_{\mu,\alpha}$ is a bounded operator from $\mathscr{B}$ into $\mathscr{B}_{\alpha-1}$ if and only if $\mu$ is a 1-logarithmic 2-Carleson measure.
		\item[$(ii)$] $\mathcal{H}_{\mu,\alpha}$ is a compact operator from $\mathscr{B}$ into $\mathscr{B}_{\alpha-1}$ if and only if $\mu$ is a vanishing 1-logarithmic 2-Carleson  measure.
	\end{itemize}
\end{theorem}

By applying Corollary \ref{cor32}, Proposition \ref{pro32} and Theorem \ref{th23}, we can obtain the following.
\begin{theorem}
	Suppose that $\alpha\geq 2$ and $1<\beta<\infty$. Then
	\begin{itemize}
		\item[$(i)$] $\mathcal{H}_{\mu,\alpha}$ is a bounded operator from $\mathscr{B}_{\beta}$ into $\mathscr{B}_{\alpha-1}$ if and only if $\mu$ is a ($\beta$+1)-Carleson measure.
		\item[$(ii)$] $\mathcal{H}_{\mu,\alpha}$ is a compact operator from $\mathscr{B}$ into $\mathscr{B}_{\alpha-1}$ if and only if $\mu$ is a vanishing ($\beta$+1)-Carleson measure.
	\end{itemize}
\end{theorem}

The final aim is to give a better necessary condition for some cases.

For $0< p<\infty$, a fuction $f\in H(\mathbb{D})$ is said to be of the class $\mathcal{Q}_{p}$ (cf. \cite{Xiao-Q}) in case
$$\sup_{w\in\mathbb{D}}\int_{\mathbb{D}}|f'(z)|^{2}(-\log|\varphi_{w}(z)|)^{p}dA(z)<\infty,$$
where $\varphi_{w}(z)=\frac{w-z}{1-\overline{w}z}$ is a special M$\ddot{\text{o}}$bius map that interchanges the points $0$ and $w$.

\begin{lemma} $\cite{Xiao-Q}$\label{le32}
	Let $0<p<\infty$ and let $f(z)=\sum_{k=0}^{\infty}a_kz^k$ with $a_k$ nonnegative and nonincreasing. Then $f\in\mathcal{Q}_{p}$ if and only if $\sup_{k}ka_{k}<\infty.$
\end{lemma}
\begin{theorem}
	Suppose $0<\alpha\leq 1$, $0<\beta<1$ and $0<p<\infty$. If $\mathcal{H}_{\mu,\alpha}$ is a bounded operator from $\mathscr{B}_{\beta}$ into $\mathcal{Q}_{p}$, then the measure $\mu$ is an $\alpha$-Carleson measure.
\end{theorem}
\begin{proof}
	Assume that	$\mathcal{H}_{\mu,\alpha}$ is a bounded operator from $\mathscr{B}_{\beta}$ into $\mathcal{Q}_{p}$. We set $f(z)=1\in\mathscr{B_{\beta}}$. Then
	$$	\mathcal{H}_{\mu,\alpha}(f)(z)=\sum_{n=0}^{\infty}\big(\sum_{k=0}^{\infty}\mu_{n,k,\alpha}a_k\big)z^n=\sum_{n=0}^{\infty}\mu_{n,0,\alpha}z^n\in \mathcal{Q}_{p}.$$
	Notice that $\mu_{n,0,\alpha}=\int_{[0,1)}\frac{\Gamma(n+\alpha)}{n!\Gamma(\alpha)}t^{n}d\mu(t)$ ($0<\alpha\leq 1$) is positive and decreasing and $\frac{\Gamma(n+\alpha)}{n!\Gamma(\alpha)}\sim n^{\alpha-1}$ as $n\rightarrow \infty$ by Stirling's formula. For any $0<\lambda<1$, we take $n$ with $1-\frac{1}{n+1}\leq \lambda<1-\frac{1}{n}$. By Lemma \ref{le31}, we obtain that
	$$\infty>n\mu_{n,0,\alpha}\geq Cn^{\alpha}\int_{\mathbb{D}}t^{n}d\mu(t)\geq Cn^{\alpha}\lambda^{t}\int_{\lambda}^{1}d\mu(t)\geq \frac{C\mu([\lambda,1))}{(1-\lambda)^{\alpha}}.$$
	This shows that the measure $\mu$ is an $\alpha$-Carleson measure. The proof is completed.
\end{proof}

It is well known that the class $\mathcal{Q}_{p}$ ($p>1$) is equivalent to the Bloch space $\mathscr{B}$ and $\mathcal{Q}_{1}$ is just the space $BMOA$ (see \cite{Girela-2001} for the more information of $BMOA$). So we immediately have the following corollaries.
\begin{corollary}
	Suppose $0<\alpha\leq 1$, $0<\beta<1$. If $\mathcal{H}_{\mu,\alpha}$ is a bounded operator from $\mathscr{B}_{\beta}$ into $\mathscr{B}$, then the measure $\mu$ is an $\alpha$-Carleson measure.
\end{corollary}
\begin{corollary}
	Suppose $0<\alpha\leq 1$, $0<\beta<1$. If $\mathcal{H}_{\mu,\alpha}$ is a bounded operator from $\mathscr{B}_{\beta}$ into $BMOA$, then the measure $\mu$ is an $\alpha$-Carleson measure.
\end{corollary}

\end{document}